\documentclass{amsart}

\usepackage{amssymb}
\usepackage{xcolor}
\usepackage{amsmath}
\usepackage{bm}
\usepackage{tkz-euclide}
\usepackage{graphicx}
\usepackage{caption}

\newcommand{\sums}{\mathop{{\sum}^{*}}}

\newcommand{\mz}{\ensuremath{\mathbb Z}}

\newcommand{\mq}{\ensuremath{\mathbb Q}}

\newcommand{\half}{\ensuremath{ \frac{1}{2}}}

\newcommand{\thalf}{\tfrac12}

\newcommand{\ov}{\overline}

\makeatletter
\@namedef{subjclassname@2020}{\textup{2020} Mathematics Subject Classification}
\makeatother

\theoremstyle{plain}		
	\newtheorem{mytheo}{Theorem}[section]
	
	\newtheorem{myprop}[mytheo]{Proposition}
	
     \newtheorem{mylemma}[mytheo]{Lemma}

	\newtheorem{myremark}[mytheo]{Remark}

\theoremstyle{remark}

	\makeatletter
\@namedef{subjclassname@2020}{\textup{2020} Mathematics Subject Classification}
\makeatother

\numberwithin{equation}{section}

\begin{document}

\title{Upper bounds for analytic ranks of elliptic curves over cyclotomic fields}

\author{Agniva Dasgupta}

\author{Rizwanur Khan}

\address{
Department of Mathematical Sciences\\ University of Texas at Dallas\\ Richardson, TX 75080-3021}
\email{agniva.dasgupta@utdallas.edu, rizwanur.khan@utdallas.edu}

\thanks{The authors were supported by the National Science Foundation grant DMS-2341239. The second author was also supported by NSF grant DMS-2344044.}

\begin{abstract}
Let $E$ be an elliptic curve defined over $\mq$. We show that the analytic rank of $E$ over the cyclotomic extension $\mathbb{Q}(e^{2\pi i/q})$ is bounded above by $q^{\frac{45}{52}+\varepsilon}$, as $q\to \infty$ through the primes. This improves the bound $q^{\frac{7}{8}+\varepsilon}$ established by Chinta.
 \end{abstract}

\maketitle

\section{Introduction}
\label{sec:intro}

\subsection{Statement of results}

Let $E$ be an elliptic curve defined over $\mq$. For number fields $K$, it is an interesting problem to study how the rank of $E(K)$, the group of $K$-rational points of $E$, varies as $K$ varies. This paper is concerned with the cyclotomic field extensions $K_q =\mathbb{Q}(e^{2\pi i/q})$, as $q\to \infty$ through the primes. 

The {\it analytic rank} of $E(K_q)$ is defined to be the order of vanishing at the central point of the $L$-function of $E$ over $K_q$, which by the Birch and Swinnerton-Dyer Conjecture is expected to equal the (algebraic) rank of $E(K_q)$. Let $L(E,s)$ denote the $L$-function of $E$ over $\mathbb{Q}$ and $L(E \otimes \chi, s)$ its twist by a Dirichlet character $\chi$, normalized so that the central point is $s=\half$. The $L$-function of $E$ over $K_q$ equals
\[
\prod_{\chi \bmod{q}} L(E \otimes \chi, s).
\]
 Given that $L(E \otimes \chi, \half)$ can vanish to order at most $O(\log q)$ (since there are $O(\log q)$ non-trivial zeros a bounded distance from the real axis), we have the trivial bound
\begin{align}
\label{rank-formula} \text{analytic rank of }E(K_q) = \sum_{\chi \bmod{q}} \  \underset{s=\frac12}{\text{ord}}  \ L(E \otimes \chi, s) \ll q\log q. 
\end{align}
Assuming the Generalized Riemann Hypothesis, Murty \cite{Murty} proved that the analytic rank is $O(q)$.  Chinta \cite{Chinta} proved a remarkable unconditional bound that saved a power of $q$. He showed that
\[
\text{analytic rank of }E(K_q) \ll q^{\frac78+\varepsilon},
\]
and this has remained unimproved until now. Our main theorem is
\begin{mytheo}
\label{thm:ellipticcurves}
    Let $E$ be an elliptic curve defined over $\mq$. Let $K_q$ be the the cyclotomic field obtained by adjoining the $q^{\text{th}}$ roots of unity to $\mq$, where $q$ is a prime. For any $\varepsilon>0$ and $q$ sufficiently large, we have that
    \begin{equation*}
        \text{analytic rank of }E(K_q) \ll q^{\frac{45}{52}+\varepsilon}.
    \end{equation*}
\end{mytheo}
As seen from \eqref{rank-formula}, the problem is equivalent to proving that for all but $O(q^{\frac{45}{52}+\varepsilon})$ of the $q-1$ Dirichlet characters of modulus $q$, the central value $L(E \otimes \chi, \half)$ is non-vanishing. To this end we prove the following more general result, which implies what we need since $E$ is known to be modular \cite{Wiles, WT, BCDT}.
\begin{mytheo}
 \label{thm:modforms}
Let $\varepsilon>0$. Let $f$ be a holomorphic Hecke newform of level $N$ with trivial nebentypus and rational Fourier coefficients. Let $\chi$ be a primitive Dirichlet character with modulus $q$ such that the order of $\chi$ (in $\widehat{(\mathbb{Z}/q\mathbb{Z})^*}$) is larger than $q^{\frac{45}{52}+\varepsilon}$. Then for $q$ a sufficiently large prime (depending on $\epsilon$ and $f$), we have that $L(f \otimes \chi, \tfrac12) \neq 0$. 
\end{mytheo}
\noindent Theorem \ref{thm:modforms} implies Theorem \ref{thm:ellipticcurves} because the number of Dirichlet characters mod $q$ with order less than $q^{\frac{45}{52}+\varepsilon}$ is $O(q^{\frac{45}{52}+\varepsilon})$. Here and throughout we use the $\varepsilon$-convention, using $\varepsilon$ to denote an arbitrarily small positive constant, but not necessarily the same one from one occurrence to another.

\subsection{Averaging over Galois Orbits}
\label{subsec:galoisavg}

Let $\mq[\chi]$ be the field extension of $\mq$ obtained by adjoining the elements in the image of $\chi$ and let $G= \text{Gal}(\mq[\chi]/\mq)$. For $\sigma \in G$, we have a Dirichlet character $\chi^\sigma$ (with the same order as $\chi$) defined via $\chi^\sigma(n) = \sigma(\chi(n))$. This gives rise to a group action of $G$ on the set of Dirichlet characters of modulus $q$. 

Our starting point to prove non-vanishing is, following Rohrlich \cite{Rohrlich} and Chinta \cite{Chinta}, to utilize the following result of Shimura \cite{Sh1,Sh2}.
\begin{myprop}
\label{prop:Shimura}
    Let $f$, $\chi$ and $\sigma$ be as above. We have
        \begin{equation}
       \nonumber  
          L(f \otimes \chi, \thalf) =0 \iff L(f \otimes \chi^\sigma, \thalf) =0.
    \end{equation}
\end{myprop}
\noindent Thus we have the implication \begin{equation}
\label{eq:moment}
 \frac{1}{|G|}\sum_{\sigma \in G} {L(f \otimes \chi^\sigma, \thalf)} \neq 0 \Longrightarrow L(f\otimes\chi,\thalf)\neq 0. 
\end{equation}
The greater the size of the set we are are averaging over (a Galois orbit of size equal to the order of $\chi$), the more likely it is that we will be able to understand the mean value on the left hand side of \eqref{eq:moment}. Therefore for characters of order $O(q^{1-\gamma})$, this will be a hopeless task if $\gamma$ is very close to $1$. But even for a character $\chi$ of maximum possible order $q-1$, it is not clear how to compute the mean value.

The $L$-function $L(f \otimes \chi, s)$ has analytic conductor $q^2$ in the $q$-aspect. In light of its approximate functional equation at $s=\half$, it is instructive to consider the test case 
 \begin{align}
\label{testcase} \frac{1}{|\mathcal{F}|}  \sum_{\chi \in \mathcal{F}} \ \sum_{1\le n\le q^{1+\varepsilon}} \frac{\lambda_f(n) \chi(n)}{\sqrt{n}}
 \end{align}
for different families $\mathcal{F}$, where $\lambda_f(n)$ are the Hecke eigenvalues associated to $f$. It is expected that the $n=1$ term will give rise to the main term of \eqref{testcase}. If $\mathcal{F}$ is the set of all Dirichlet characters, then
\begin{align*}
&\frac{1}{q-1}\sum_{\chi \bmod q} \  \sum_{1\le n\le q^{1+\varepsilon}} \frac{\lambda_f(n)  \chi(n)}{\sqrt{n}} =   1+O\Big(\sum_{\substack{2\le n\le q^{1+\varepsilon}\\ n\equiv 1 \bmod q}} \frac{|\lambda_f(n)|}{\sqrt{n}}\Big) =1 + O(q^{-\half+\varepsilon}).
 \end{align*}
If $\mathcal{F}$ is a Galois orbit of a primitive Dirichlet character of modulus $3^k$ as $k\to\infty$ (a simple case considered by Rohrlich), then (cf. \cite[Lemma 2.3]{KMN}) we have
 \begin{align*}
&\frac{1}{|G|}\sum_{\sigma \in G}  \  \sum_{1\le n\le q^{1+\varepsilon}} \frac{\lambda_f(n) \chi^\sigma(n)}{\sqrt{n}}
 =  1+ O\Big( \sum_{\substack{2\le n\le q^{1+\varepsilon}\\ n^2 \equiv 1 \bmod 3^{k-1}}} \frac{|\lambda_f(n)| }{\sqrt{n}}\Big) = 1+ O(3^{-\frac{k}{2}+\varepsilon}).
\end{align*}
If $\mathcal{F}$ is the Galois orbit of a Dirichlet character of prime modulus $q\to \infty$ and order $q-1$, then (cf. \eqref{chi-avg-formula}) we have
  \begin{align*}
&\frac{1}{|G|}\sum_{\sigma \in G}  \  \sum_{1\le n\le q^{1+\varepsilon}} \frac{\lambda_f(n) \chi^\sigma(n)}{\sqrt{n}}
 =  1+ \sum_{\substack{2\le n\le q^{1+\varepsilon}\\ (n,q)=1 }}  \frac{\lambda_f(n) }{\sqrt{n}}
 \frac{\mu(\mathrm{ord}(n))}{\phi(\mathrm{ord}(n))},
\end{align*}
where $\mathrm{ord}(n)$ denotes the order of $n$ in $(\mz/q\mz)^*$. In this case we must show that the contributions of the $n\ge 2$ terms exhibit cancellation, but this seems to be an intractable problem. 

Chinta's beautiful idea to get around this difficulty was to instead use the implication
 \begin{equation}
\nonumber
 \frac{1}{|G|}\sum_{\sigma \in G} {L(f \otimes \chi^\sigma, \thalf)}  M_X(f\otimes \chi^\sigma,\thalf) \neq 0 \Longrightarrow L(f\otimes\chi,\thalf)\neq 0,
\end{equation}
where $M_X(f\otimes \chi^\sigma,\thalf)$ is a mollifier, defined in \eqref{mol-def}. The mollifier comes at a cost but has the effect (cf. \eqref{mol-eff}) that
\[
 \frac{1}{|G|}\sum_{\sigma \in G}  \Big(\sum_{1\le n\le q^{1+\varepsilon}} \frac{\lambda_f(n) \chi^\sigma(n)}{\sqrt{n}} \Big)M_X(f\otimes \chi^\sigma,\thalf)\sim \frac{1}{|G|}\sum_{\sigma \in G}  1 \sim 1.
\]
However, by Lemma \ref{lem:afemol}, the approximate functional equation of $L(f\otimes\chi,\thalf)$ also contains a dual sum of the shape
\[
\varepsilon(f\otimes \chi) \sum_{1\le n\le q^{1+\varepsilon}} \frac{\lambda_f(n) \ov{\chi}(n)}{\sqrt{n}},
\]
where $\varepsilon(f\otimes \chi)$ is the root number of the $L$-function. The contribution of the dual sum must be shown to be smaller than the main term, and this is plausible due to extra cancellation from the root number. Chinta was able to prove exactly this, after lessening the effect of the dual sum by using an unbalanced approximate function equation that is shorter on the dual side.

Our work picks up from Chinta's set-up. Our new contribution is a superior treatment of the dual sum contribution. This boils down to obtaining non-trivial cancellation in a certain sum of Kloosterman sums (where Chinta was only able to obtain pointwise bounds using Weil's bound for the Kloosterman sums). For this, we use an extension of ideas that were developed in the work of the second author and Ngo \cite{KN} on the non-vanishing of Dirichlet $L$-functions.

\section{The work of Chinta}
All of the background material in this section, except for a small part of Lemma \ref{av-results}, is already established in \cite{Chinta}. Throughout the rest of the paper, suppose that $q$ is a large enough prime so that $(q,N)=1$. Let
\begin{equation*}
     \varepsilon(f\otimes \chi) = \varepsilon(f) \chi(N) \frac{\tau(\chi)^2}{q},
 \end{equation*}
where $\tau(\chi)$ denotes the Gauss sum and $|\varepsilon(f)|=1$.
Let $X=q^b$ and $Y=q^a$ for some $0<b<a<1$ to be chosen later. There exist constants $a_n, c_n \ll n^\varepsilon$ depending on $f$ such that the Dirichlet polynomial
\begin{align}
\label{mol-def} M_X(f\otimes \chi,s) = \sum_{n\le X} \frac{c_n \chi(n)}{n^s}
\end{align}
satisfies for $\Re(s)>1$,
\begin{align}
\label{mol-eff} L(f\otimes \chi,s) M_X(f\otimes \chi,s) = 1+ \sum_{n>X}  \frac{a_n \chi(n)}{n^s}.
\end{align}
For more details, see \cite[Section 3]{Chinta}.

At $s=\half$, we have the following approximate functional equation.
\begin{mylemma}
\label{lem:afemol}
    We have
    \begin{multline*}
        L(f \otimes \chi, \thalf)M_X(f \otimes \chi,\thalf) =1 + \sum_{n > X} \frac{a_n \chi(n)}{\sqrt{n}} \mathrm{exp}\Big(\frac{-2\pi n}{qY\sqrt{N}}\Big) \\ + \varepsilon(f\otimes \chi) \sum_{n \geq 1} \sum_{1 \leq m \leq X} \frac{\lambda_f(n)c_m \ov{\chi}(n)\chi(m)}{\sqrt{nm}}\mathrm{exp}\Big(\frac{-2\pi n Y}{mq\sqrt{N}}\Big) + O(q^{-1}).
    \end{multline*}
    \end{mylemma}
    \proof
    This is established in \cite[Section 4]{Chinta}.
    \endproof

Defining
\[
 \chi_{\text{av}}(n) := \frac{1}{|G|} \sum_{\sigma \in G} \chi^\sigma(n)
 \]
 and
 \[
     \tilde{\chi}_{\text{av}}(n) := \frac{1}{|G|} \sum_{\sigma \in G} \varepsilon(f \otimes \chi^\sigma) \ov{\chi^\sigma}(n),
     \]
 we have the following pointwise and average evaluations of these sums.
 \begin{mylemma}
 \label{av-results}
Let $q$ be a prime. Let $\chi$ be a primitive Dirichlet character of modulus $q$ such that the order of $\chi$ is $\frac{\phi(q)}{d}$ for some $d|\phi(q)$. We have for $(nN,q)=1$,
\begin{align}
\label{chi-avg-formula} {\chi}_{\mathrm{av}}(n)=\frac{\mu(\mathrm{ord}(n^d))}{\phi(\mathrm{ord}(n^d))},
\end{align}
and
\begin{align}
\label{tilde-chi-av}     \tilde{\chi}_{\mathrm{av}}(n) = \frac{\varepsilon(f)}{q} \sum_{r\bmod q} {\chi}_{\mathrm{av}}(rN) S(r  , n, q),
     \end{align}
     where $S(r , n, q)$ is the Kloosterman sum,
     and
     \begin{align}
 \label{chi-avg} q^{-\varepsilon} d \ll  \sum_{r \bmod q} |{\chi}_{\mathrm{av}}(r) | =  \sum_{r \bmod q} |{\chi}_{\mathrm{av}}(rN) |\ll  q^\varepsilon d.
     \end{align}
    \end{mylemma}
    \proof This may be found in the statement and proof of \cite[Proposition 1]{Chinta}, except for the lower bound in \eqref{chi-avg}, which we will show here. But first we note that in the upper bound in \eqref{chi-avg}, Chinta has $\sigma_1(d)2^{\nu(\phi(q))}$ while we have $dq^\varepsilon$, where $\sigma_1(d)$ is the sum of the positive divisors of $d$ and $\nu(\phi(q))$ is the number of distinct prime divisors of $\phi(q)$. Our replacement is valid because $\sigma_1(d)=d\sum_{k|d} \frac{1}{k}\ll q^\varepsilon d$ and $2^{\nu(\phi(q))}\ll 2^{\frac{\log q}{\log \log q}}\ll q^\varepsilon$. 
 
To prove the lower bound in \eqref{chi-avg}, note that
\[
 \sum_{r \bmod q} |{\chi}_{\mathrm{av}}(r) |=  \sum_{1\le r <q} |{\chi}_{\mathrm{av}}(r) | = \sum_{1\le r <q} \frac{1}{\phi(\mathrm{ord}(r^d))}\ge \sum_{\substack{1\le r <q\\ \mathrm{ord}(r)=d }} 1 =\phi(d) \gg d^{1-\varepsilon}.\qedhere
\]
 \endproof
 
 By Lemma \ref{lem:afemol}, we can write
 \[
  \frac{1}{|G|}\sum_{\sigma \in G} {L(f \otimes \chi^\sigma, \thalf)}  M_X(f\otimes \chi^\sigma,\thalf)= 1+ S_1 +S_2 +o(1), 
  \]
 where
 \[
 S_1= \sum_{n > X} \frac{a_n \chi_{\text{av}}(n)}{\sqrt{n}} \mathrm{exp}\Big(\frac{-2\pi n}{qY\sqrt{N}}\Big)
 \]
 and
 \[
S_2= \sum_{n\ge 1} \sum_{1 \leq m \leq X} \frac{\lambda_f(n)c_m  \tilde{\chi}_{\text{av}}(n\ov{m}) }{\sqrt{nm}}\mathrm{exp}\Big(\frac{-2\pi n Y}{mq\sqrt{N}}\Big).
 \]
Our goal is to show that $S_1=o(1)$ and $S_2=o(1)$, which we achieve if the order of $\chi$ is not too small (equivalently, in the notation below, $\gamma$ is not too close to 1). 

For $S_1$, we have the following bound.
  \begin{myprop}
 \label{S1-result}
 Let $q$ be a prime. Let $\chi$ be a primitive Dirichlet character of modulus $q$ such that the order of $\chi$ is $\frac{\phi(q)}{d}$ for some $d|\phi(q)$.  Writing $d=q^\gamma$, suppose that $\gamma>0$. For any $a+1<c<2$, we have
\[
    S_1 \ll q^\varepsilon \cdot \mathrm{max}(q^{\gamma-\tfrac{b}{2}},q^{\gamma+\tfrac{c}{2}-1}) + \mathrm{exp}\Big(-\frac{q^{c-1-a-\varepsilon}}{\sqrt{N}}\Big).
 \]
    \end{myprop}
    \proof 
    This may be found in \cite[page 22]{Chinta} and arises as follows. For $X<n< q^c$, $S_1$ is bounded trivially using \eqref{chi-avg}. For $n\ge q^c$, $S_1$ is bounded trivially using the exponential decay of $\mathrm{exp}(\frac{-2\pi n}{qY\sqrt{N}})$.
    \endproof 
    
Now we turn to $S_2$. Using \eqref{tilde-chi-av}, we have
\begin{align}
\label{s2exp}
S_2\ll \frac{1}{q}\Big| \sum_{ 1\le n \le \frac{q^{1+\varepsilon}X}{Y} } \ \sum_{1 \leq m \leq X} \ \sum_{r \bmod q} {\chi}_{\mathrm{av}}(rN) \frac{\lambda_f(n)c_m }{\sqrt{nm}}\mathrm{exp}\Big(\frac{-2\pi n Y}{mq\sqrt{N}}\Big)   S(r , n \ov{m}, q) \Big| + O(q^{-100}),
\end{align}
where the error arises from the truncation of the $n$-sum. Note that $(nm,q)=1$ in these sums (in fact $n<q^{1-\varepsilon}, m<q^{1-\varepsilon}$ since $0<b<a<1$). Chinta bounded $S_2$ by moving the absolute value signs to the inside, applying \eqref{chi-avg}, and applying Weil's bound for the Kloosterman sums. We will obtain a superior bound by proving cancellation between the Kloosterman sums.
 
 \section{Proof of Theorem \ref{thm:modforms}}
 Our improvement over Chinta's work is contained in the following bound for $S_2$, which will be proved in the next section.
  \begin{myprop}
 \label{S2-result}
Let $q$ be a prime. Let $\chi$ be a primitive Dirichlet character of modulus $q$ such that the order of $\chi$ is $\frac{\phi(q)}{d}$ for some $d|\phi(q)$.  Write $d=q^\gamma$ and suppose that $0<2b<a<1$. We have
 \[
 S_2 \ll q^{\varepsilon + \frac{3b}{4} -  \frac{3a}{8}  -\frac{1}{16}+\gamma} +q^{\varepsilon + \frac{3b}{4} -  \frac{3a}{8}  +\frac{\gamma}{2}}.
 \]
    \end{myprop}
 Thus to show that $S_1=o(1)$ and $S_2=o(1)$, we need
 \begin{align}
 \label{constraints} \gamma-\frac{b}{2}<0,    \gamma+\frac{c}{2}-1<0,   a+1<c<2,    \frac{3b}{4}-\frac{3a}{8}-\frac{1}{16}+\gamma <0, \frac{3b}{4}-\frac{3a}{8}+\frac{\gamma}{2} <0,  0 < 2b < a <1.
 \end{align}
 for $0<\gamma< 1$ as large as possible. One can check that for any $\gamma < \frac{7}{52}-\varepsilon$, we can satisfy \eqref{constraints} by choosing 
 \begin{align}
 \label{choices} b = 2(\tfrac{7}{52}), \ \ c=2 -2(\tfrac{7}{52}), \text{ and } \ \ a=1-2(\tfrac{7}{52})-\varepsilon,
 \end{align} 
where the two occurrences of $\varepsilon$ in \eqref{choices} represent the same small constant. A computer search verifies that $\frac{7}{52}$ is the optimal value (up to $\varepsilon$) of $\gamma$. We include the Mathematica code for this below. The condition $\gamma < \frac{7}{52}-\varepsilon$ translates to the order of $\chi$ being larger than $q^{\frac{45}{52}+\varepsilon}.$ This proves Theorem \ref{thm:modforms}.

 \begin{verbatim}

In[1]:= Maximize[{gamma, gamma - (b/2) <= 0 &&  gamma + (c/2) - 1 <= 0 && 
        a + 1 <= c <= 2 &&  (3b/4) - (3a/8) - (1/16) + gamma <= 0 && 
        (3b/4) - (3a/8) + (gamma/2) <= 0 && 0 <= 2b <= a <= 1}, {gamma, a, b, c}]

Out[1]= {7/52, {gamma -> 7/52, a -> 19/26, b -> 7/26, c -> 45/26}}

\end{verbatim}

\section{Proof of Proposition \ref{S2-result}}

The following two lemmas are extensions of \cite[Lemmas 3.1, 3.2]{KN}.
\begin{mylemma}
\label{lem:31mod}
   Let $q$ be a prime and $k$ a positive integer. Let $L_j = \sum_{1\leq r \leq R} |z_r|^j$ denote the $j^{\textrm{th}}$ moment of the sequence $(z_r)_{1 \leq r \leq R}$. Suppose that
   \begin{align}
   \label{moments-condition}
   L_j \le L_1, \ \ \ L_1\ge 1. 
   \end{align} For $R<q$, we have   
    \begin{align}
\label{sumofproduct}        \sum_{1 \leq r_1, r_2, \cdots, r_{2k} \leq R} \ |z_{r_1}| & \cdots |z_{r_{2k}}| \Big| \sums_{h \bmod{q}}  S(h,{r_1},q) \cdots S(h,{r_{2k}},q)\Big| \ll L_1^{k}q^{k+1} + L_1^{2k} q^{k+\frac12},
    \end{align}
    where $\sums$ restricts to the non-zero residue classes.
    
\end{mylemma}

\begin{proof}
Let $\mathfrak{D}$ be the set consisting of all tuples $(r_1,r_2,\cdots,r_{2k})$ such that no component $r_i$ is distinct from the others (that is, for each $i$ there is a $j \neq i$ with $r_i=r_j$). We split the sum on the left side accordingly as
\[
\sum_{(r_1,\cdots,r_{2k}) \in \mathfrak{D}} + \sum_{(r_1,\cdots,r_{2k}) \not \in \mathfrak{D}}.
\]
For the sum over tuples in $\mathfrak{D}$, we use the Weil bound to bound the product of Kloosterman sums and trivially bound the $h$-sum. Thus
\[
\sum_{(r_1,\cdots,r_{2k}) \in \mathfrak{D}} \ll q^{k+1} \sum_{(r_1,\cdots,r_{2k}) \in \mathfrak{D}}  |z_{r_1}|\cdots |z_{r_{2k}}|.
\]
Grouping together identical coefficients $|z_{r_i}|$ and observing that each group has at least two elements, we have
\[
\sum_{(r_1,\cdots,r_{2k}) \in \mathfrak{D}} |z_{r_1}| \cdots |z_{r_{2k}}| \ll \sum_{\substack{  j_1, \ldots, j_\ell \ge 2\\ j_1+\ldots +j_\ell =2k}} L_{j_1}\cdots L_{j_\ell} \ll L_1^\ell \ll L_1^{k},
\]
where we used \eqref{moments-condition}.

For the sum over tuples not in $\mathfrak{D}$, we have 
\[
\sums_{h \bmod{q}}  S(h,{r_1},q) \cdots S(h,{r_{2k}},q)\ll q^{k+\half}
\]
by \cite[Proposition 3.2]{fgkm} (which reflects square-root cancellation once the Kloosterman sums are opened up). Thus
\[
\sum_{(r_1,\cdots,r_{2k}) \notin \mathfrak{D}} \ll q^{k+1} \sum_{(r_1,\cdots,r_{2k}) \notin \mathfrak{D}}  |z_{r_1}|\cdots |z_{r_{2k}}|
\]
Grouping together identical coefficients $|z_{r_i}|$ and observing that groups may have only one element, we get \[
\sum_{(r_1,\cdots,r_{2k}) \notin \mathfrak{D}} |z_{r_1}|\cdots |z_{r_{2k}}| \ll  L_1^{2k}.
\qedhere
\]
\end{proof}

\begin{mylemma}
\label{lem:32mod}
    Let $q$ be a prime and let
    \begin{equation*}
        \mathcal{S} = \sum_{\substack{ n\leq N, m\le M, r\le R }}x_{n,m} z_r S(n\ov{m},r,q),
    \end{equation*}
    where $x_{n,m} \ll q^{\varepsilon}$ and the moments of $z_r$ satisfy \eqref{moments-condition}.  Suppose that $NM < q$ and $R<q$. For any integer $k>1$ we have the upper bound
    \begin{equation*}
        \mathcal{S} \ll q^{\varepsilon} (NM)^{1-\frac{1}{2k}} (L_1q^{\half+\frac{1}{4k}} + L_1^\frac12 q^{\half + \frac{1}{2k}} ).
    \end{equation*}
    As before $L_j$ denotes the $j^{\textrm{th}}$ moment of the sequence $(z_r)_{1 \leq r \leq R}$.
\end{mylemma}

\begin{proof}

Using Holder's inequality with the exponents $(\frac{k}{k-1},k)$, we have

\begin{align*}
    \mathcal{S} &\ll \Big( \sum_{n\le N, m\le M}|x_{n,m}|^{\frac{k}{k-1}}\Big)^{\frac{k-1}{k}}  \Big(\sum_{n\le N, m\le M }\Big|\sum_{1\leq r \leq R} z_rS(n\ov{m},r,q)\Big|^k \Big)^{\frac{1}{k}}\\
    &\ll  q^\varepsilon (NM)^{1-\frac{1}{k}}  \Big(\sums_{h \bmod q } v(h) \Big|\sum_{1\leq r \leq R} z_r S(h,r,q)\Big|^k \Big)^{\frac{1}{k}},
\end{align*}
where
 \begin{equation*}
    v(h) = \sum_{\substack{N \le N, m\le M}} \delta(n\ov{m}\equiv h \bmod q),
 \end{equation*}
and $\delta(P)=1$ if the statement $P$ is true and  $\delta(P)=0$ otherwise. By Cauchy's inequality, we have
\begin{equation}
\label{eq:mathcals3}
    \mathcal{S} \ll q^\varepsilon (NM)^{1-\frac{1}{k}}\Big( \sums_{h \bmod q} v(h)^2\Big)^{\frac{1}{2k}} \Big(\sums_{h \bmod q }  \Big|\sum_{1\leq r \leq R} z_r S(h,r,q)\Big|^{2k} \Big)^{\frac{1}{2k}}.
\end{equation}
We have
\begin{align}
\label{vbound}    \sums_{h \bmod q} v(h)^2 &=\sum_{ \substack{n_1, n_2 \leq N \\ m_1,m_2\le M}} \delta(n_1\ov{m_1}\equiv n_2\ov{m_2} \bmod q) = \sum_{ \substack{n_1, n_2 \leq N \\ m_1,m_2\le M}} \delta(n_1m_2 \equiv n_2m_1\bmod q) \ll q^\varepsilon NM,
\end{align}
using the condition $NM<q$ to see that $n_1m_2 \equiv n_2m_1\bmod q$ implies $n_1m_2=n_2m_1$.
 For the last sum in \eqref{eq:mathcals3}, we have
 \begin{align}
  \nonumber  &\sums_{h \bmod q}\Big|\sum_{1\leq r \leq R}z_rS(h,r,q)\Big|^{2k} = \sums_{h \bmod q}\Big(\sum_{1\leq r \leq R}z_rS(h,r,q)\Big)^{k}\overline{ \Big(\sum_{1\leq r \leq R}z_rS(h,r,q)\Big)^{k}}\\
  \label{eq:hsumbound}   &\ll \sum_{1 \leq r_1, r_2, \cdots, r_{2k} \leq R} \ |z_{r_1}|\cdots |z_{r_{2k}}| \Big| \sums_{h \bmod{q}}  S(h,{r_1},q) \cdots S(h,{r_{2k}},q)\Big| \ll L_1^{k}q^{k+1} + L_1^{2k} q^{k+\frac12},
 \end{align}
 using first that the Kloosterman sums are real valued, and then Lemma \ref{lem:31mod} to get the final bound. Inserting the bounds \eqref{eq:hsumbound} and \eqref{vbound} into \eqref{eq:mathcals3} completes the proof.  \end{proof} 

We are now ready to prove Proposition \ref{S2-result}. Dividing the right hand side of \eqref{s2exp} into dyadic intervals, we have
\begin{align*}
S_2\ll \max_{\substack{1\le N \le \frac{q^{1+\varepsilon}X}{Y} \\ 1\le M\le X}} \  \frac{1}{q^{1-\varepsilon} (NM)^{\half} }\Big| \sum_{ \frac{N}{2}\le n\le N} \ \sum_{ \frac{M}{2}\le m\le M} \ \sum_{1\le r < q} x_{n,m} z_r  S(r , n \ov{m}, q) \Big|+O(q^{-100}),
\end{align*}
where
\[
x_{n,m}=\lambda_f(n)c_m \mathrm{exp}\Big(\frac{-2\pi n Y}{mq\sqrt{N}}\Big), \ \ \ z_r = {\chi}_{\mathrm{av}}(rN).
\]
The condition $2b<a$ from the statement of Proposition \ref{S2-result} ensures that $NM<q$ is satisfied. We note that \eqref{moments-condition} is satisfied because $L_j\le L_1$ since $|{\chi}_{\mathrm{av}}(rN)|\le 1$ and $L_1\ge 1$ by the lower bound in \eqref{chi-avg}. By Lemma \ref{lem:32mod} and the upper bound $L_1\ll q^\varepsilon d$ from \eqref{chi-avg}, we get 
\begin{align*}
S_2\ll \max_{\substack{1\le N \le \frac{q^{1+\varepsilon}X}{Y} \\ 1\le M\le X}} \  q^{-1+\varepsilon} (NM)^{\half -\frac{1}{2k}} (d q^{\half+\frac{1}{4k}} + d^\frac12 q^{\half + \frac{1}{2k}} ).
\end{align*}
The maximum occurs at the greatest possible values of $N$ and $M$. Thus using the notation $X=q^b, Y=q^a, d=q^\gamma$, we have
\begin{align}
\label{s2boundk} S_2 \ll q^{\varepsilon + b(1-\frac{1}{k}) - a(\half-\frac{1}{2k}) -\frac{1}{4k}+\frac{\gamma}{2}}(q^{\frac{\gamma}{2}}+q^{\frac{1}{4k}}).
\end{align}
Chinta already showed in \cite[Theorem 3]{Chinta} that $S_1=o(1)$ and $S_2=o(1)$ for $\gamma < \frac{1}{8}$. The two terms in \eqref{s2boundk} balance when $\frac{\gamma}{2}=\frac{1}{4k}$, or $k=\frac{1}{2\gamma}$, provided that $k>1$ is an integer, which for $\gamma=\frac18$ gives $k=4$. This suggests that to improve Chinta's work, we should specialize to $k=4$. Doing so gives
\begin{align*}
S_2 \ll q^{\varepsilon + \frac{3b}{4} -  \frac{3a}{8}  -\frac{1}{16}+\frac{\gamma}{2}}(q^{\frac{\gamma}{2}}+q^{\frac{1}{16}}),
\end{align*}
as required.

\begin{myremark}
The choice $k=4$ means that we are utilizing cancellation in a sum of a product of 8 Kloosterman sums, as seen in \eqref{sumofproduct}. In contrast, \cite{KN} uses cancellation in a sum of a product of 4 Kloosterman sums. If we had followed \cite{KN} verbatim, we would have recovered Chinta's result and not improved it. Thus our extension of \cite{KN} is necessary: a product of 6 Kloosterman sums would have already given an improvement, but a product of 8 Kloosterman sums yields the best result by this method.
\end{myremark}

\bibliographystyle{amsplain} 

\bibliography{bibliography}

\end{document}